\documentclass[preprint,12pt]{elsarticle}

\usepackage{amssymb,latexsym,amsmath,epsfig,amsthm,mathrsfs}
\usepackage{amsfonts}
\usepackage{amscd}
\usepackage{dsfont}
\usepackage{bbm}
\usepackage{rotating}
\usepackage{newlfont}
\usepackage{enumitem}

\newtheorem{theorem}{Theorem}

\newtheorem{remark}{Remark}
\numberwithin{theorem}{section}

\usepackage{lineno,hyperref}
\modulolinenumbers[5]

\newtheoremstyle{case}{}{}{}{}{}{:}{ }{}
\theoremstyle{case}
\newtheorem{case}{Case}
\newtheoremstyle{subcase}{}{}{}{}{}{:}{ }{}
\theoremstyle{subcase}
\newtheorem{subcase}{Subcase}

\begin{document}

\begin{frontmatter}

\title{Direct and Inverse Theorems on Signed Sumsets of Integers}

\author[mymainaddress]{Ram Krishna Pandey\corref{myauthor}}
\cortext[myauthor]{Corresponding author}
\ead{ramkpandey@gmail.com}

\author[mymainaddress]{Jagannath Bhanja\corref{mycorrespondingauthor}}
\cortext[mycorrespondingauthor]{Research supported by the Ministry of Human Resource Development, India}
\ead{jagannathbhanja74@gmail.com}

\address[mymainaddress]{Department of Mathematics, Indian Institute of Technology Roorkee, Uttarakhand, 247667, India}
%\address[mysecondaryaddress]{Department of Mathematics, Indian Institute of Technology Patna, Patna 800013, India}

\begin{abstract}
Let $G$ be an additive abelian group and $h$ be a positive integer. For a nonempty finite subset $A=\{a_0, a_1,\ldots, a_{k-1}\}$ of $G$, we let
\[h_{\underline{+}}A:=\{\Sigma_{i=0}^{k-1}\lambda_{i} a_{i}: (\lambda_{0}, \ldots, \lambda_{k-1}) \in \mathbb{Z}^{k},~ \Sigma_{i=0}^{k-1}|\lambda_{i}|=h \},\]
be the {\it signed sumset} of $A$.

The {\it direct problem} for the signed sumset $h_{\underline{+}}A$ is to find a nontrivial lower bound for $|h_{\underline{+}}A|$ in terms of $|A|$. The {\it inverse problem} for $h_{\underline{+}}A$ is to determine the structure of the finite set $A$ for which $|h_{\underline{+}}A|$ is minimal. In this article, we solve both the direct and inverse problems for $|h_{\underline{+}}A|$, when $A$ is a finite set of integers.
\end{abstract}


\end{frontmatter}

%\linenumbers

\section{Introduction}\label{intro}

Let $G$ be an additive abelian group and $A=\{a_0, a_1,\ldots, a_{k-1}\}$ be a nonempty finite subset of $G$. Let $h$ be a positive integer. The $h$-fold sumset $hA$ of $A$ is the set of all sums of $h$ elements of $A$, that is,
\[hA=\{\Sigma_{i=0}^{k-1}\lambda_{i} a_{i}: (\lambda_{0}, \ldots, \lambda_{k-1}) \in \mathbb{N}_{0}^{k}, \Sigma_{i=0}^{k-1}\lambda_{i}=h \}.\]
The $h$-fold signed sumset of $A$, denoted by $h_{\underline{+}}A$, is defined by
\[h_{\underline{+}}A:=\{\Sigma_{i=0}^{k-1}\lambda_{i} a_{i}: (\lambda_{0}, \ldots, \lambda_{k-1}) \in \mathbb{Z}^{k}, \Sigma_{i=0}^{k-1}|\lambda_{i}|=h \}.\]
Clearly,
\begin{equation}
hA\cup h(-A) \subseteq h_{\underline{+}}A \subseteq h\left(A\cup -A\right)\nonumber,
\end{equation}
and for any integer $\alpha$,
\begin{equation}
h_{\underline{+}}(\alpha *A)=\alpha * (h_{\underline{+}}A),\nonumber
\end{equation}
where \[\alpha*A=\{\alpha*a:~a\in A\},\]
is the $\alpha$-dilation of the set $A$.

The study of sumsets and hence of multiple fold sumsets of sets of an additive abelian group has more than two-hundred-year old history. The sumsets are actually the foundations of the ``additive number theory". A paper of Cauchy \cite{cauchy} in 1813, which is believed to be one of the oldest and classical work off-course, finds the minimum cardinality of the sumset $A+B$, where $A$ and $B$ are nonempty subsets of residue classes modulo a prime. Later, Davenport \cite{dav} rediscovered Cauchy's result in 1935. The result is now known as the Cauchy-Davenport theorem. Several partial results about the minimum cardinality of the sumsets and its inverse that if the minimum cardinality is achieved, then the characterization of individual sets have been obtained by now. A comprehensive list of references may be found in Mann \cite{mann}, Freiman \cite{freiman}, Nathanson \cite{N96}, and  Tao \cite{tao}. Plagne \cite{plagne} in 2006, finally settled the general case by obtaining the minimum cardinality of sumset in an abelian group.

In contrast to the $h$-fold sumset, the $h$-fold signed sumset has a brief and a quite young history. This sumset first appeared in the work of Bajnok and Ruzsa \cite{BR03} in the context of the ``independence number" of a subset $A$ of $G$ and in the work of Klopsch and Lev \cite{KL03, KL09} in the context of the ``diameter" of $G$ with respect to the subset $A$. The first systematic and point centric study appeared in the work of Bajnok and Matzke \cite{BM15} in which they studied the minimum cardinality of $h$-fold signed sumset $h_{\underline{+}}A$ of subsets of a finite abelian group. In particular, they proved that the minimum cardinality of $h_{\underline{+}}A$ is the same as the minimum cardinality of $hA$, when $A$ is a subset of a finite cyclic group. An year later, they \cite{BM2015} classified all possible values of $k$ for which the minimum cardinality of $h_{\underline{+}}A$ coincide with the minimum cardinality of $hA$, when $A$ is a subset of a particular elementary abelian group.

The direct problem for signed sumset $h_{\underline{+}}A$ is to find a nontrivial lower bound for $|h_{\underline{+}}A|$ in terms of $|A|$. The inverse problem for $h_{\underline{+}}A$ is to determine the structure of the finite set $A$ for which $|h_{\underline{+}}A|$ is minimal. In this article, we study both direct and inverse problems for signed sumset $h_{\underline{+}}A$, when $A$ is a finite set of integers. This study is done in Section \ref{sec2} by considering three different cases, viz.; (i) A contains only positive integers, (ii) A contains positive integers and zero, and (iii) A contains arbitrary integers, in the subsections \ref{subsec2.1}, \ref{subsec2.2}, and \ref{subsec2.3}, respectively. To prove our results, we need the following classical results about $h$-fold sumset $hA$.

\begin{theorem} \cite{N96}\label{thmA}
Let $h\geq 2$ and let $A$ be a finite set of integers with $|A|=k$. Then
\[|hA| \geq hk-h+1.\]
\end{theorem}

\begin{theorem} \cite{N96}\label{thmB}
Let $h\geq 2$ and let $A$ be a finite set of integers with $|A|=k$. Then
$|hA|=hk-h+1$ if and only if $A$ is a $k$-term arithmetic progression.
\end{theorem}

\section{Direct and inverse theorems for $h_{\underline{+}}A$}\label{sec2}

\subsection{$A$ contains only positive integers}\label{subsec2.1}

\begin{theorem}\label{thm1}
Let $h$ be a positive integer and let $A$ be a finite set of $k$ positive integers. We have
\[|h_{\underline{+}}A| \geq 2(hk-h+1).\] This lower bound is best possible for $h\leq2$.
\end{theorem}

\begin{proof}
Let $A=\{a_0, a_1,\ldots, a_{k-1}\}$, where $0 < a_{0} < a_{1} < \cdots < a_{k-1}$. The sumset $h_{\underline{+}}A$ contains at least the following $2(hk-h+1)$ integers.
\begin{eqnarray}\label{eqn1}
ha_{0} &<& (h-1)a_{0}+a_{1} < (h-2)a_{0}+2a_{1} < \cdots < a_{0}+(h-1)a_{1} < ha_{1}\nonumber \\
 &<& (h-1)a_{1}+a_{2} < (h-2)a_{1}+2a_{2} < \cdots < a_{1}+(h-1)a_{2} < ha_{2}\nonumber \\
 & \vdots & \nonumber\\
 &<& (h-1)a_{k-2}+a_{k-1} < (h-2)a_{k-2}+2a_{k-1} < \cdots < a_{k-2}+(h-1)a_{k-1}\nonumber \\
 &<& ha_{k-1}
\end{eqnarray}
and
\begin{eqnarray}\label{eqn2}
-ha_{k-1} &<& -(h-1)a_{k-1}-a_{k-2} < \cdots < -a_{k-1}-(h-1)a_{k-2} < -ha_{k-2}\nonumber \\
 &<& -(h-1)a_{k-2}-a_{k-3} < \cdots < -a_{k-2}-(h-1)a_{k-3} < -ha_{k-3}\nonumber \\
 & \vdots & \nonumber\\
 &<& -(h-1)a_{1}-a_{0} < \cdots < -a_{1}-(h-1)a_{0} < -ha_{0}.
\end{eqnarray}

Thus, \[|h_{\underline{+}}A| \geq 2(hk-h+1).\]
Next, we show that this lower bound is best possible. If $h=1$, then $|1_{\underline{+}}A|=2k$. Hence the lower bound is tight for every finite set $A$. Next, let $h=2$ and $A=\{1, 3, 5, \ldots, 2k-1 \}$. Then \[2_{\underline{+}}A=\{-(4k-2), \ldots, -4, -2, 2, 4, \ldots, (4k-2)\}.\]
 zHence, $|2\underline{+}A|=4k-2$. This completes the proof of the theorem.
\end{proof}

\begin{theorem}\label{thm2}
Let $h \geq 2$ and let $A$ be a finite set of $k$ positive integers. If $|h_{\underline{+}}A|=2(hk-h+1)$, then $h=2$ and $A=d*\{1,3,\ldots,2k-1\}$, for some positive integer $d$.
\end{theorem}

\begin{proof}
Let $A=\{a_0, a_1,\ldots, a_{k-1}\}$, where $0 < a_{0} < a_{1} < \cdots < a_{k-1}$.
Since $|h_{\underline{+}}A|=2(hk-h+1)$, it follows from Theorem \ref{thm1}, that the sumset $h_{\underline{+}}A$ consists precisely the integers listed in  (\ref{eqn1}) and (\ref{eqn2}). For each $i=1, 2, \ldots, k-2$, we have
\[a_{i-1}+(h-1)a_{i} < ha_{i} < (h-1)a_{i}+a_{i+1}.\]
Also,
 \[a_{i-1}+(h-1)a_{i} < a_{i-1}+(h-2)a_{i}+a_{i+1} < (h-1)a_{i}+a_{i+1}.\]
Thus,
 \[ha_{i}=a_{i-1}+(h-2)a_{i}+a_{i+1}.\]
 This is equivalent to \[a_{i+1}-a_{i}=a_{i}-a_{i-1}.\]
 Therefore, the set $A$ is in arithmetic progression, i.e., $a_{i}-a_{i-1}=d$, for some $d>0$ and for all $1\leq i\leq k-1$. Again,
\begin{eqnarray}\label{eqn3}
-ha_{1} &<& -(h-1)a_{1}-a_{0}\nonumber \\ &<& -(h-1)a_{1}+a_{0}\nonumber \\ &<& -(h-2)a_{1}+2a_{0}\\ & \vdots &\nonumber \\ &<& -a_{1}+(h-1)a_{0}\nonumber \\ &<& ha_{0}\nonumber.
\end{eqnarray}
Thus, from (\ref{eqn1}), (\ref{eqn2}) and (\ref{eqn3}), it follows that, for $i=1,2,\ldots, h-1$,
\[-(h-i)a_{1}+ia_{0}=-(h-i-1)a_{1}-(i+1)a_{0}.\]
So, the common difference $d=a_{1}-a_{0}=2ia_{0}$,
for $i=1,2,\ldots, h-1$. This is possible, only if $h=2$. Hence, \[A=d*\{1,3,\ldots, 2k-1\}.\]
This completes the proof of the theorem.
\end{proof}

\begin{theorem}\label{thm3}
Let $h\geq 3$ be a positive integer. Let $A$ be a finite set of $k\geq 3$ positive integers. Then
\begin{equation}\label{eqn4}
|h_{\underline{+}}A| \geq 2hk-h+1.
\end{equation}
This lower bound is best possible.
\end{theorem}

The above theorem does not hold for $k=2$, as it can be seen by taking $A=\{1,2\}$, $h=3$; $A=\{1,3\}$, $h=4$; and $A=\{2,3\}$, $h=5$.

Further, if $A=\{a_{0},a_{1}\}$, where $0<a_{0}<a_{1}$ and $h<\frac{a_{0}+a_{1}}{2a_{0}}$, we observe in the following remark that $|h_{\underline{+}}A|=4h$.

\begin{remark}
Let $h\geq 3$ and $A=\{a_{0},a_{1}\}$, where $0<a_{0}<a_{1}$. Then, every summand in $h_{\underline{+}}A$ is either of the form $(h-i)a_{0}+ia_{1}$, or $(h-i)a_{0}-ia_{1}$, or $-(h-i)a_{0}+ia_{1}$, or $-(h-i)a_{0}-ia_{1}$, where $0\leq i\leq h$. Hence, the maximum possibility of integers in $h_{\underline{+}}A$ is $4h$, i.e.,
\begin{equation}\label{eqnr1}
|h_{\underline{+}}A|\leq 4h.
\end{equation}
On the other hand, as $h<\frac{a_{0}+a_{1}}{2a_{0}}$, i.e., $0<(2h-1)a_{0}<a_{1}$, we have
\begin{eqnarray*}
ha_{0}&<&-(h-1)a_{0}+a_{1}<(h-1)a_{0}+a_{1}<-(h-2)a_{0}+2a_{1}\\ &<& (h-2)a_{0}+2a_{1}< \cdots
                                                               < -a_{0}+(h-1)a_{1}<a_{0}+(h-1)a_{1}<ha_{1}.
\end{eqnarray*}
Since each of the above $2h$ signed $h$-fold summand is positive and in $h_{\underline{+}}A$, their negatives are also in $h_{\underline{+}}A$. Hence, $|h_{\underline{+}}A|\geq 4h$. This together with (\ref{eqnr1}) give $|h_{\underline{+}}A|=4h$.
\end{remark}

%Further observe that $|h_{\underline{+}}A|<4h$, if
%\begin{enumerate}
%\item $h=4$ and $A=\{1,2\}$, $\{1,5\}$, and $\{1,7\}$.
%\item $h=5$ and $A=\{1,4\}$, $\{1,5\}$.
%\end{enumerate}
%Notice that in both the above cases we have $h\geq \frac{a_{0}+a_{1}}{2a_{0}}$.

\begin{proof}[\textbf{Proof of theorem \ref{thm3}.}]
Let $A=\{a_0, a_1,\ldots, a_{k-1}\}$ be a finite set of integers, where $0 < a_{0} < a_{1} < \cdots < a_{k-1}$. From Theorem \ref{thm1}, it follows that the sumset $h_{\underline{+}}A$ contains at least $2(hk-h+1)$ integers listed in (\ref{eqn1}) and (\ref{eqn2}). So, it remains to show at least $(h-1)$ extra integers in $h_{\underline{+}}A$ different from the integers  in (\ref{eqn1}) and (\ref{eqn2}). To show this, we consider three cases depending on $a_{2}-a_{1}<a_{1}-a_{0}$, $a_{2}-a_{1}>a_{1}-a_{0}$, and $a_{2}-a_{1}=a_{1}-a_{0}$. Except in a subcase of the last case, namely, $a_{2}-a_{1}=a_{1}-a_{0}=2a_{0}$, which will lead to present the example for the best possible bound, we show much more extra summands than $h-1$ in $h_{\underline{+}}A$.

\begin{case}\label{case1} ($a_{2}-a_{1}<a_{1}-a_{0}$, i.e., $a_{2}<2a_{1}-a_{0}$). Consider the following sequence of integers, which is taken from (\ref{eqn1}).
\begin{equation}\label{eqn5}
(h-1)a_{0}+a_{1}<(h-2)a_{0}+2a_{1}<(h-3)a_{0}+3a_{1}<\cdots<a_{0}+(h-1)a_{1}<ha_{1}
\end{equation}

We shall insert an extra signed $h$-fold summand between each pair of successive integers of (\ref{eqn5}) as follows:

\noindent $(h-1)a_{0}+a_{1}<(h-1)a_{0}+a_{2}<(h-2)a_{0}+2a_{1}<(h-2)a_{0}+a_{1}+a_{2}<(h-3)a_{0}+3a_{1}<(h-3)a_{0}+2a_{1}+a_{2}<(h-4)a_{0}+4a_{1}
<\cdots<2a_{0}+(h-3)a_{1}+a_{2}<a_{0}+(h-1)a_{1}<a_{0}+(h-2)a_{1}+a_{2}<ha_{1}$.

Thus, we get $h-1$ extra positive integers of $h_{\underline{+}}A$. Similarly, taking the negatives of these $h-1$ summands, we get another set of $h-1$ integers of $h_{\underline{+}}A$. Hence, we get a total of at least $2(h-1)$ extra integers of $h_{\underline{+}}A$, not already listed in (\ref{eqn1}) and (\ref{eqn2}).
\end{case}

%\vspace{0.5 cm}

\begin{case}\label{case2} ($a_{2}-a_{1}>a_{1}-a_{0}$, i.e., $2a_{1}<a_{2}+a_{0}$). Similar to the Case 1, we have

\noindent $ha_{1}<(h-2)a_{1}+a_{2}+a_{0}<(h-1)a_{1}+a_{2}<(h-3)a_{1}+2a_{2}+a_{0}<(h-2)a_{1}+2a_{2}<(h-4)a_{1}+3a_{2}+a_{0}<(h-3)a_{1}+3a_{2}
<\cdots<(h-1)a_{2}+a_{0}<a_{1}+(h-1)a_{2}$.

So, we get $h-1$ extra summands in $h_{\underline{+}}A$ between $ha_{1}$ and $a_{1}+(h-1)a_{2}$. Hence, taking negatives of these $h-1$ positive summands, we get a total of at least $2(h-1)$ extra integers of $h_{\underline{+}}A$.
\end{case}

%\vspace{0.5 cm}

\begin{case}\label{case3} ($a_{2}-a_{1}=a_{1}-a_{0}$, i.e., $a_{0}$, $a_{1}$, $a_{2}$ are in arithmetic progression). Let $a_{1}=a_{0}+d$, $a_{2}=a_{0}+2d$, for some positive integer $d$.

\begin{subcase}\label{subcase1} ($d>2a_{0}$). Consider the following integers of (\ref{eqn1})

$ha_{0}<(h-1)a_{0}+a_{1}<(h-2)a_{0}+2a_{1}<\cdots<a_{0}+(h-1)a_{1}<ha_{1}<(h-1)a_{1}+a_{2}<(h-2)a_{1}+2a_{2}<\cdots<a_{1}+(h-1)a_{2}<ha_{2}$.
\\ Rewrite the list as

$ha_{0}<ha_{0}+d<ha_{0}+2d<\cdots<ha_{0}+(h-1)d<ha_{0}+hd<ha_{0}+(h+1)d<ha_{0}+(h+2)d<\cdots<ha_{0}+(2h-1)d<ha_{0}+2hd$.

For each $i=0,1,\ldots,h-2$, we insert an extra summand between $ha_{0}+2id$ and $ha_{0}+(2i+1)d$. We have,

$ha_{0}+2id<(h-2)a_{0}+(2i+1)d=(h-2-i)a_{0}-a_{1}+(i+1)a_{2}<ha_{0}+(2i+1)d$.

Each of these $h-1$ extra signed $h$-fold summands $(h-2-i)a_{0}-a_{1}+(i+1)a_{2}$, is positive. So, we get $h-1$ extra positive integers of $h_{\underline{+}}A$. The negatives of these $h-1$ integers are also signed $h$-fold summands, hence are in the set $h_{\underline{+}}A$ and different from the summands in (\ref{eqn2}). Hence, we get at least $2(h-1)$ extra integers of $h_{\underline{+}}A$, which are not listed in (\ref{eqn1}) and (\ref{eqn2}).
\end{subcase}

\begin{subcase}\label{subcase2} ($d<2a_{0}$). We use induction argument on $h$ to write $\lfloor\frac{h}{2}\rfloor$ extra positive integers of $h_{\underline{+}}A$.

\noindent If $h=3$, then \[a_{0}<a_{2}-a_{1}+a_{0}=a_{0}+d<3a_{0}.\]
If $h=4$, then \[2a_{0}<a_{2}-a_{1}+2a_{0}=2a_{0}+d<4a_{0},\] and \[0<-a_{1}+3a_{0}=2a_{0}-d<2a_{0}.\]
If $h=5$, then \[3a_{0}<a_{2}-a_{1}+3a_{0}=3a_{0}+d<5a_{0},\] and \[a_{0}<-a_{1}+4a_{0}=3a_{0}-d<3a_{0}.\]
If $h=6$, then \[4a_{0}<a_{2}-a_{1}+4a_{0}=4a_{0}+d<6a_{0},\] \[2a_{0}<-a_{1}+5a_{0}=4a_{0}-d<4a_{0},\] and \[0<2a_{2}-3a_{1}+a_{0}=d<2a_{0}.\]

In all the above cases we get exactly $\lfloor\frac{h}{2}\rfloor$ number of extra positive signed $h$-fold summands,  which are not included in (\ref{eqn1}) and (\ref{eqn2}). Now, let $h\geq 7$ and assume that the result is true for $h-1$. If $h=4k+1$ or $h=4k+3$ for some $k\geq 1$, then  $\lfloor\frac{h}{2}\rfloor=\lfloor\frac{h-1}{2}\rfloor=\frac{h-1}{2}$. By the induction hypothesis, $\lfloor\frac{h-1}{2}\rfloor$ extra positive integers as signed $(h-1)$-fold summands may be obtained in $(h-1)_{\underline{+}}A$. Adding a single copy of $a_{0}$ to all these $(h-1)$-fold summands, we can obtain $\lfloor\frac{h-1}{2}\rfloor(=\lfloor\frac{h}{2}\rfloor)$ extra positive signed $h$-fold summands. This completes the induction in this case.

 Now, let $h=4k$, $k\geq 1$. Then $\lfloor\frac{h-1}{2}\rfloor$ extra positive integers may be obtained from the $\lfloor\frac{h-1}{2}\rfloor$ extra positive summands of $(h-1)$-fold signed sumset of $A$ by just adding $a_{0}$ to it and one more summand is given by $0<(k-1)a_{2}-(2k-1)a_{1}+(k+2)a_{0}=2a_{0}-d<2a_{0}$. Hence, we get $\lfloor\frac{h}{2}\rfloor$ extra positive integers.

Similarly, if $h=4k+2$, $k\geq 1$, then $\lfloor\frac{h-1}{2}\rfloor$ extra positive integers may be obtained from the $\lfloor\frac{h-1}{2}\rfloor$ extra positive summands of $(h-1)$-fold signed sumset of $A$ by just adding $a_{0}$ to it and one more summand is given by $0<(k+1)a_{2}-(2k+1)a_{1}+ka_{0}=d<2a_{0}$.

Since, the negatives of these $\lfloor\frac{h}{2}\rfloor$ integers are also in the set $h_{\underline{+}}A$. Hence, we get a total of at least $2\lfloor\frac{h}{2}\rfloor$ extra integers in $h_{\underline{+}}A$.

Further, in both the above subcases \ref{subcase1} and \ref{subcase2}, we get even more $2\lfloor\frac{h}{3}\rfloor$ integers. Let $m$ be the largest integer such that $3m\leq h$, i.e., $m=\lfloor\frac{h}{3}\rfloor$ or $h=3m+\epsilon$, $\epsilon\in\{0,1,2\}$. Then,
\begin{center}
\[(h-3)a_{0}+2a_{1}-a_{2}=(h-2)a_{0},\]
\[(h-6)a_{0}+4a_{1}-2a_{2}=(h-4)a_{0},\]
\[(h-9)a_{0}+6a_{1}-3a_{2}=(h-6)a_{0},\]
\vdots
\[\epsilon a_{0}+2ma_{1}-ma_{2}=(m+\epsilon)a_{0}.\]
\end{center}

So, there are $m=\lfloor\frac{h}{3}\rfloor$ further extra positive signed $h$-fold summands which are multiples of $a_{0}$, between $0$ and $ha_{0}$. Thus, including negatives of these integers we get, $2m=2\lfloor\frac{h}{3}\rfloor$ even more extra integers in both the subcases $d>2a_{0}$ and $d<2a_{0}$. Hence, in both the subcases \ref{subcase1} and \ref{subcase2}, we get a total of at least $2(\lfloor\frac{h}{2}\rfloor + \lfloor\frac{h}{3}\rfloor)$ extra signed $h$-fold summands neither included in (\ref{eqn1}) nor in (\ref{eqn2}).
\end{subcase}

\begin{subcase}\label{subcase3} ($d=2a_{0}$). In this case we show that  $-(h-2)a_{0},-(h-4)a_{0},\\-(h-6)a_{0},\ldots,(h-6)a_{0},(h-4)a_{0},(h-2)a_{0}$ are signed $h$-fold summands, which are neither included in (\ref{eqn1}) nor in (\ref{eqn2}). Clearly, their number is $h-1$.

If $h=3$, then $2a_{1}-a_{2}=a_{0}$, and $a_{2}-2a_{1}=-a_{0}$. So, we get $(h-1)=2$ distinct integers which are previously not included.

Now, let $h\geq 4$.  Rewrite the summands of (\ref{eqn1}), which are between $ha_{0}$ and $ha_{1}$ as follows:
\begin{equation}\label{eqn6}
(h-1)a_{0}+a_{1}<(h-2)a_{0}+2a_{1}<(h-3)a_{0}+3a_{1}<\cdots<a_{0}+(h-1)a_{1}
\end{equation}

 Adding $-(a_{1}+a_{2})$ to the first three successive integers $(h-1)a_{0}+a_{1}$, $(h-2)a_{0}+2a_{1}$, $(h-3)a_{0}+3a_{1}$ of (\ref{eqn6}), we get
\[(h-1)a_{0}+a_{1}-(a_{1}+a_{2})=(h-1)a_{0}-a_{2}=(h-6)a_{0},\]
\[(h-2)a_{0}+2a_{1}-(a_{1}+a_{2})=(h-2)a_{0}+a_{1}-a_{2}=(h-4)a_{0},\]
and
\[(h-3)a_{0}+3a_{1}-(a_{1}+a_{2})=(h-3)a_{0}+2a_{1}-a_{2}=(h-2)a_{0}.\]

Now leave the first term of (\ref{eqn6}) and add $-2(a_{1}+a_{2})$ to the next three successive integers $(h-2)a_{0}+2a_{1}$, $(h-3)a_{0}+3a_{1}$, $(h-4)a_{0}+4a_{1}$ of (\ref{eqn6}), we get
\[(h-2)a_{0}+2a_{1}-2(a_{1}+a_{2})=(h-2)a_{0}-2a_{2}=(h-12)a_{0},\]
\[(h-3)a_{0}+3a_{1}-2(a_{1}+a_{2})=(h-3)a_{0}+a_{1}-2a_{2}=(h-10)a_{0},\]
and
\[(h-4)a_{0}+4a_{1}-2(a_{1}+a_{2})=(h-4)a_{0}+2a_{1}-2a_{2}=(h-8)a_{0}.\]

We continue this process up to the last triplet $3a_{0}+(h-3)a_{1}$, $2a_{0}\\+(h-2)a_{1}$, $a_{0}+(h-1)a_{1}$ of (\ref{eqn6}) by adding $-(h-3)(a_{1}+a_{2})$, to get
\[3a_{0}+(h-3)a_{1}-(h-3)(a_{1}+a_{2})=3a_{0}-(h-3)a_{2}=-(5h-18)a_{0},\]
\[2a_{0}+(h-2)a_{1}-(h-3)(a_{1}+a_{2})=2a_{0}+a_{1}-(h-3)a_{2}=-(5h-20)a_{0},\]
and
\[a_{0}+(h-1)a_{1}-(h-3)(a_{1}+a_{2})=a_{0}+2a_{1}-(h-3)a_{2}=-(5h-22)a_{0}.\]

The above process covers all the $h-1$ integers $-(h-2)a_{0},-(h-4)a_{0},\\-(h-6)a_{0}, \ldots,(h-6)a_{0},(h-4)a_{0},(h-2)a_{0}$ as signed $h$-fold summands with some other possible negative integers which are already counted in (\ref{eqn2}). One may stop this process till one gets $-(h-2)a_{0}$. Thus, we get exactly $h-1$ extra integers of $h_{\underline{+}}A$, not already included in (\ref{eqn1}) and (\ref{eqn2}).
\end{subcase}
\end{case}

Thus, in all the above cases \ref{case1}, \ref{case2} and \ref{case3}, we get at least $h-1$ extra integers of $h_{\underline{+}}A$, which are not included in (\ref{eqn1}) and (\ref{eqn2}). Hence, $|h_{\underline{+}}A| \geq 2hk-h+1$.

Next, we show that this lower bound is best possible. Let $A=\{1,3,5,\ldots,\\(2k-1)\}$ for some integer $k\geq 1$. If $h$ is even, then
\[h_{\underline{+}}A \subseteq \{-h(2k-1), \ldots, -4, -2,0, 2, 4, \ldots, h(2k-1)\}.\]
If $h$ is odd, then \[h_{\underline{+}}A \subseteq \{-h(2k-1), \ldots, -5, -3, -1, 1, 3, 5, \ldots, h(2k-1)\}.\]

In both these cases, $|h_{\underline{+}}A|\leq 2hk-h+1$. Hence, together with (\ref{eqn4}), we get, $|h_{\underline{+}}A|=2hk-h+1$.
This completes the proof of the theorem.
\end{proof}

\begin{theorem}\label{thm4}
Let $h \geq 3$ and let $A$ be a finite set of $k\geq 3$ positive integers. If $|h_{\underline{+}}A|=2hk-h+1$, then $A=d*\{1,3,\ldots,2k-1\}$, for some positive integer $d$.
\end{theorem}

\begin{proof}
Let $A=\{a_0, a_1,\ldots, a_{k-1}\}$, where $0 < a_{0} < a_{1} < \cdots < a_{k-1}$.
Since $|h_{\underline{+}}A|=2hk-h+1$, it follows from Theorem \ref{thm3}, that $a_{2}-a_{1}=a_{1}-a_{0}=d=2a_{0}$. Again, by the similar argument used in Theorem \ref{thm2}, we get, for each $i=1, 2, \ldots, k-2$

\[a_{i-1}+(h-1)a_{i} < ha_{i} < (h-1)a_{i}+a_{i+1},\]
and
 \[a_{i-1}+(h-1)a_{i} < a_{i-1}+(h-2)a_{i}+a_{i+1} < (h-1)a_{i}+a_{i+1}.\]
Thus,
 \[ha_{i}=a_{i-1}+(h-2)a_{i}+a_{i+1}.\]
 This is equivalent to \[a_{i+1}-a_{i}=a_{i}-a_{i-1}.\]
 Therefore, the set $A$ is in arithmetic progression, and hence \[A=d*\{1,3,\ldots, 2k-1\}.\]
This completes the proof of the theorem.
\end{proof}

\subsection{$A$ contains nonnegative integers with $0\in A$}\label{subsec2.2}

\begin{theorem}\label{thm5}
Let $h\geq 1$. Let $A$ be a finite set of $k$ nonnegative integers with $0\in A$. Then
\begin{equation}\label{eqn7}
|h_{\underline{+}}A| \geq 2hk-2h+1.
\end{equation}
This lower bound is best possible.
\end{theorem}

\begin{proof}
Let $A=\{a_0, a_1,\ldots, a_{k-1}\}$, where $0 = a_{0} < a_{1} < \cdots < a_{k-1}$.
From (\ref{eqn1}) and (\ref{eqn2}), it is clear that $h_{\underline{+}}A$ contains at least $hk-h$ positive integers $(h-i)a_{j}+ia_{j+1}$ and $hk-h$ negative integers $-(h-i)a_{j}-ia_{j+1}$, for $0 \leq i \leq h$, $ 1 \leq j \leq k-2$, and one extra integer zero. Thus,
\[|h_{\underline{+}}A| \geq 2hk-2h+1.\]

Next, we show that this lower bound is best possible. Let $A=\{0, 1, 2, \ldots,\\ k-1\}=[0,k-1]$. The smallest integer of $h_{\underline{+}}A$ is $-h(k-1)$ and the largest element of $h_{\underline{+}}A$ is $h(k-1)$. Therefore, \[h_{\underline{+}}A \subseteq [-h(k-1), h(k-1)].\]
So \[|h_{\underline{+}}A| \leq 2h(k-1)+1.\] This inequality together with (\ref{eqn7}), implies \[|h_{\underline{+}}A|=2hk-2h+1.\]
This completes the proof of the theorem.
\end{proof}

\begin{theorem}\label{thm6}
Let $h \geq 2$ and let $A$ be a finite set of $k$ nonnegative integers with $0\in A$. Then $|h_{\underline{+}}A|=2hk-2h+1$ if and only if $A=d*[0,k-1]$, for some positive integer $d$.
\end{theorem}

\begin{proof}
Let $A=\{a_0, a_1,\ldots, a_{k-1}\}$, where $0 = a_{0} < a_{1} < \cdots < a_{k-1}$.
Since $|h_{\underline{+}}A|=2hk-2h+1$, it follows from Theorem \ref{thm5} that the sumset $h_{\underline{+}}A$ consists precisely the integers listed in equation (\ref{eqn1}) and (\ref{eqn2}). By the similar argument as used in Theorem \ref{thm2} and Theorem \ref{thm4}, we obtain that the set $A$ is an arithmetic progression. Hence $A=d*[0,k-1]$.
 This completes the proof of the theorem.
\end{proof}

\subsection{$A$ contains positive, negative and zero}\label{subsec2.3}

\begin{theorem}\label{thm7}
Let $h\geq 1$ and let $A$ be a finite set of $k$ integers contains both positive and negative integers. Then
\begin{equation}\label{eqn8}
|h_{\underline{+}}A| \geq hk-h+1.
\end{equation}
This lower bound is best possible.
\end{theorem}
\begin{proof}
The lower bound is trivial and it follows from (\ref{eqn1}). To see that the lower bound is optimal, consider the interval of integers $A=\left[-\lfloor \frac{k}{2} \rfloor, \lfloor \frac{k}{2} \rfloor\right ]$, where $k\geq 3$ is an odd integer. Then,
\[h_{\underline{+}}A \subseteq \left [-h\lfloor \frac{k}{2} \rfloor, h\lfloor \frac{k}{2} \rfloor\right ].\]
Thus, \[|h_{\underline{+}}A| \leq 2h\lfloor \frac{k}{2} \rfloor+1=(k-1)h+1=hk-h+1.\]
This inequality together with (\ref{eqn8}) gives $|h_{\underline{+}}A|=hk-h+1$. This completes the proof of the theorem.
\end{proof}

\begin{theorem}\label{thm8}
Let $A$ be a finite set of $k\geq 2$ integers. Let $|h_{\underline{+}}A|=hk-h+1$. Then $A$ is a symmetric set and it is an arithmetic progression.
\end{theorem}

\begin{proof}
Let $A=\{a_0, a_1,\ldots, a_{k-1}\}$, where $a_{0} < a_{1} < \cdots < a_{k-1}$. Let $|h_{\underline{+}}A|=hk-h+1$. Since $hA \subseteq h_{\underline{+}}A$, Theorem \ref{thmA} implies that $hA=h_{\underline{+}}A$. Thus, by Theorem \ref{thmB} the set $A$ is in arithmetic progression. Again, since $|h_{\underline{+}}A|=hk-h+1$, the sumset $h_{\underline{+}}A$ contains precisely that $(hk-h+1)$ integers listed in (\ref{eqn1}). It also contains the $(hk-h+1)$ integers listed in (\ref{eqn2}). Thus, for all $i=0,1,\ldots,k-1$, we have
\[ha_{i}=-ha_{k-1-i}.\]
 This is equivalent to $a_{i}= -a_{k-1-i}$, for all $i=0,1,\ldots,k-1$. This completes the proof of the theorem.
\end{proof}

\section*{Acknowledgements} The first author would like to thank to the Indian Institute of Technology Roorkee for providing the grant to carry out the research with Grant No. MAT/FIG/100656.

%Optional place for funding acknowledgements and/or expression of gratitude.

%\makeatletter
%\renewcommand{\@biblabel}[1]{[#1]\hfill}
%\makeatother
%\clearpage

\section*{References}
\bibliographystyle{model1-num-names}
\bibliography{sample.bib}

\begin{thebibliography}{99}

\bibitem{BM15} B. Bajnok, R. Matzke, The minimum size of signed sumsets, \textit{The Electronic Journal of combinatorics} {\bf 22} (2) (2015), P2.50.

\bibitem{BM2015} B. Bajnok, R. Matzke, On the minimum size of signed sumsets in elementary abelian groups, \textit{Journal of Number Theory} {\bf 159} (2015) 384--401.

\bibitem{BR03} B. Bajnok, R. Ruzsa, The Independence Number of a Subset of an Abelian Group, \textit{integers} {\bf 3} (2003), Paper No. A2, 23 pp.

\bibitem {cauchy} A. L. Cauchy, Recherches sur les nombres, {\it J. \'{E}cole polytech.\/} {\bf 9} (1813) 99--116.

\bibitem {dav} H. Davenport, On the addition of residue classes, {\it J. Lond. Math. Soc.\/} {\bf 10} (1935) 30--32.

\bibitem{freiman} G. A. Freiman, Foundations of Structural Theory of Set Addition, vol {\bf 37} of Translations of Mathematical Monographs. American Mathematical Society, Provedience, R.I., 1973.

\bibitem{KL03} B. Klopsch, V. F. Lev, How long does it take to generate a group$?$, \textit{J. Algebra} {\bf 261} (2003) 145--171.

\bibitem{KL09} B. Klopsch, V. F. Lev, Generating abelian groups by addition only, \textit{Forum Math.} {\bf 21} (1) (2009) 23--41.

\bibitem{mann} H. B. Mann, Addition Theorems: the Addition Theorems of Group Theory and Number Theory, Wiley-Interscience, New York, 1965.

\bibitem{N96} M. B. Nathanson, Additive Number Theory: Inverse problems and the geometry of sumsets, Springer, 1996.

\bibitem {plagne} A. Plagne, Optimally small sumsets in groups, I. The supersmall sumset property, the $\mu_{G}^{(k)}$ and the $\nu_{G}^{(k)}$ functions, {\it Unif. Distrib. Theory,\/} {\bf 1} (2006), no. 1, 27--44.

\bibitem {tao} T. Tao, V. H. Vu, Additive Combinatorics, Cambridge studies in advanced mathematics, vol. {\bf 105}, Cambridge, 2010.


\end{thebibliography}

\end{document}